\documentclass[11pt,a4paper]{amsart}
\usepackage{amsfonts,amsmath,amssymb,amsthm}
\usepackage{enumerate}
\usepackage[utf8]{inputenc}
\usepackage{bbm,dsfont}

\usepackage{mathtools}
\mathtoolsset{showonlyrefs=true}

\DeclareSymbolFont{SY}{U}{psy}{m}{n}
\DeclareMathSymbol{\emptyset}{\mathord}{SY}{'306}

\newtheorem{theorem}{Theorem}{\bf}{\it}
\newtheorem{lemma}[theorem]{Lemma}{\bf}{\it}

\newtheorem{condition}{Condition}

\DeclareMathOperator{\argmax}{arg max}

\theoremstyle{definition}  

\title[Eplett's theorem for self--converse generalised tournaments]{Eplett's theorem for self--converse generalised tournaments}
\author[Erik ~Th\"ornblad]{Erik Th\"ornblad}
 \address{Department of Mathematics, Uppsala University, Box 480, S-75106 Uppsala, Sweden.}
 \email{erik.thornblad@math.uu.se}
 \date{\today}

\begin{document}
\begin{abstract}
The converse of a tournament is obtained by reversing all arcs. If a tournament is isomorphic to its converse, it is called self--converse. Eplett provided a necessary and sufficient condition for a sequence of integers to be realisable as the score sequence of a self--converse tournament. In this paper we extend this result to generalised tournaments. 
\end{abstract}
\maketitle

\section{Introduction and results}\label{sec:intro}
A \emph{generalised tournament} $G=(V(G),\alpha)$ is a set $V(G)=\{1,\dots, n\}$ of vertices along with a function $\alpha:V(G)\times V(G) \to [0,1]$, such that $\alpha(i,j)+\alpha(j,i)=1$ for all $(i,j)\in V(G)\times V(G)$, $i\neq j$, and $\alpha(i,i)=0$ for all $i\in V(G)$. If $G=(V(G),\alpha)$ and $\alpha\in \{0,1\}$, then we say that $G$ is a (non--generalised) \emph{tournament}. Given a vertex $i\in V(G)$, the \emph{outdegree} of $i$ is defined as $d_i=\sum_{j\in V(G)}\alpha(i,j)$. The sequence $(d_i)_{i=1}^n$ of outdegrees of $G$ is called the \emph{score sequence} of $G$.

A natural question is to ask for a condition that characterises those sequences which can be realised as the score sequence of some generalised tournament.

\begin{condition}
A sequence $(d_i)_{i=1}^n$ of non--negative real numbers is said to satisfy condition \textbf{I} if 
\begin{align}
 \sum_{i\in J} d_i \geq \binom{|J|}{2}
\end{align}
for all $J\subseteq \{1,2,\dots, n\}$, with equality for $J=\{1,\dots, n\}$. 
\end{condition}

If $(d_i)_{i=1}^n$ is the score sequence of some tournament, then one can easily see that condition \textbf{I} must be satisfied, for the subtournament induced by $J$ must have at least $\binom{|J|}{2}$ edges. One of the classical results in graph theory is the sufficiency of condition \textbf{I}, i.e. showing that if $(d_i)_{i=1}^n$ satisfies condition \textbf{I}, then there is a tournament with score sequence $(d_i)_{i=1}^n$. More precisely, Landau \cite{Landau1953} showed that a non--decreasing sequence $(d_i)_{i=1}^n$ consisting of non--negative \emph{integers} is the score sequence of some tournament if and only if condition \textbf{I} is satisfied. Subsequently Moon \cite{Moon63} extended this to the setting of generalised tournaments, showing that a non--decreasing sequence $(d_i)_{i=1}^n$ consisting of non--negative \emph{reals} is the score sequence of some generalised tournament if and only if condition \textbf{I} is satisfied. 

In this paper we study a related problem for the class of self--converse generalised tournaments. Two generalised tournaments $G_1=(V(G_1),\alpha_1)$ and $G_2=(V(G_2),\alpha_2)$ are \emph{isomorphic} if there exists a bijection $\rho : V(G_1) \to V(G_2)$ such that $\alpha_1(i,j)=\alpha_2(\rho(i),\rho(j))$ for all $i,j\in V(G_1)$.  The \emph{converse} of a generalised tournament $G=(V(G),\alpha)$ is the tournament $G'=(V(G),\alpha')$ where $\alpha'(i,j)=1-\alpha(i,j)$ for all $i,j\in V(G)$. One should think of the converse $G'$ as being obtained by reversing all arcs of $G$. A generalised tournament $G$ is \emph{self--converse} if $G$ and $G'$ are isomorphic. 

The following condition is central in the study of which sequences are realisable by self--converse generalisted tournaments.

\begin{condition}
A non--decreasing sequence $(d_i)_{i=1}^n$ of non--negative real numbers is said to satisfy condition \textbf{II} if 
\begin{align}
 d_i+d_{n+1-i} = n-1
\end{align}
for all $i=1,2,\dots, n$.
\end{condition}

It is a two--line argument that any self--converse tournament must have a score sequence satisfying condition \textbf{II}. Eplett proved sufficiency, but only for non--generalised tournaments.

\begin{theorem}[\cite{Eplett1979}] \label{thm:Eplett}
 A non--decreasing sequence $(d_i)_{i=1}^n$ of non--negative integers is the score sequence of some self--converse (non--generalised) tournament if and only if conditions \textbf{I} and \textbf{II} are satisfied.
\end{theorem}

As we shall show, Eplett's result does extend in the natural way to real sequences and self--converse generalised tournaments. The following is our main result.

\begin{theorem} \label{thm:real}
  A non--decreasing sequence $(d_i)_{i=1}^n$ of non--negative real numbers is the score sequence of some self--converse generalised tournament if and only if conditions \textbf{I} and \textbf{II} are satisfied.
\end{theorem}

In the remainder of this section, we will outline the ideas behind the proof of Theorem \ref{thm:real}. (Indeed, after reading the introduction, hopefully one should be able to fill in the missing details.) The details follow in Section \ref{sec:proofs}. It should be mentioned that the ideas are very similar to those in \cite{Thornblad2016c}, in which Moon's result is derived from Landau's result, but some technical details differ.

The proof is carried out in two steps. First, an extension to the case when the score sequence is rational, then to the case when it is real.

\begin{lemma}\label{lem:rational}
  A non--decreasing sequence $(d_i)_{i=1}^n$ of non--negative rational numbers is the score sequence of some self--converse generalised tournament if and only conditions \textbf{I} and \textbf{II} are satisfied.
\end{lemma}

The idea behind the proof of Lemma \ref{lem:rational} can be described as a  ``blow--up followed by a shrink--down''. More precisely, given a rational non--decreasing sequence $(d_i)_{i=1}^n$ satisfying conditions \textbf{I} and \textbf{II},  we consider instead a related sequence containing $mn$ integral elements, where $m$ is chosen so that $md_i$ is integral for all $i=1,\dots, n$. We show that this sequence satisfies the conditions of Theorem \ref{thm:Eplett}, so there exists a self--converse tournament $H$ having this sequence as its score sequence. After this we will divide the $mn$ vertices of $H$ into $n$ clusters of $m$ vertices. Each cluster will correspond to a vertex in a generalised tournament $G$, the edge weights between the vertices of which are obtained by averaging over the edge weights between the corresponding clusters in $H$. Finally we show that that the score sequence of $G$ is indeed $(d_i)_{i=1}^n$ and that $G$ is self--converse.

In order to carry out the extension to real sequences, we need the following approximation result.

\begin{lemma}\label{lem:tech}
 Let $(d_i)_{i=1}^n$ be a non--decreasing sequence of non--negative reals satisfying conditions \textbf{I} and \textbf{II}. Then there exist non--decreasing sequences $(d_i^{(m)})_{i=1}^n$ of non--negative rationals satisfying conditions \textbf{I} and \textbf{II} (for each $m\geq 1$) such that $d_i^{(m)} \to d_i$ as $m\to \infty$, for each $i=1,2, \dots, n$.
\end{lemma}

Let us make two observations which are helpful in the proof of Lemma \ref{lem:tech}. First, since we assume that the sequence be non--decreasing, condition \textbf{I} need only be checked for $J=\{1,2,\dots, k\}$ for $k=1,2,\dots, n$. Second, if condition \textbf{II} is satisfied, then condition \textbf{I} need only be checked for $J=\{1,2,\dots, k\}$ for $k=1,2,\dots, \lfloor n/2 \rfloor$. These observations simplify the proof; the idea is then to do a small perturbation of the sequence $(d_i)_{i=1}^n$ so that it becomes rational, taking care not to disturb the validity of condition \textbf{I} or \textbf{II}.

Given Lemma \ref{lem:rational} and Lemma \ref{lem:tech}, the proof of Theorem \ref{thm:real} is not difficult. Given a real sequence  $(d_i)_{i=1}^n$, we will approximate it by rational sequences $(d_i^{(m)})_{i=1}^n$ as in Lemma \ref{lem:tech}. By Lemma \ref{lem:rational} we can find generalised self--converse tournaments on $n$ vertices with rational edge weights and scores sequences $(d_i^{(m)})_{i=1}^n$. The final step is to note that the set of edge weights is compact, so we may select a subsequence of the generalised tournaments such that all edge weights converge. The limit object will be a well--defined self--converse generalised tournament with score sequence $(d_i)_{i=1}^n$.

\section{Proofs}\label{sec:proofs}

\begin{proof}[Proof of Lemma \ref{lem:rational}]
  Let $(d_i)_{k=1}^n$ be a non--decreasing sequence of rational numbers satisfying conditions \textbf{I} and \textbf{II}. Since the $d_i$ are rational, there exist $k_i,m_i\in \mathbb{N}$ (with no common factors) such that $d_i=k_i/m_i$. Denote by $m$ the lowest common multiple of $m_i$. (If some $k_i=0$, we may take $m_i=1$; this may happen for at most one $i$.)
 
Let us assume that $m,n$ are both odd; the other cases require only minor modifications and are left to the reader. We first construct an $n\times m$--array which will contain the outdegrees of our blow--up. For $i=1,\dots, n$ and $\ell=1,\dots, m$, let
\begin{align}
c_{i,\ell} = md_i+\frac{m-1}{2}.
\end{align}
(For $m$ even, we can let the second term be $m/2$ for $\ell=1,2,\dots, m/2$ and $m/2+1$ for $\ell= m/2+1, \dots, m$.) Since we assume that the sequence $(d_i)_{i=1}^n$ be non--decreasing, also $c_{i,\ell}$ is non--decreasing in $i$. It is clear that $c_{i,\ell}\in \mathbb{N}$ for all $i=1,\dots, m$ and $\ell=1,\dots, n$. 

The fact that the $c_{i,\ell}$ satisfy condition \textbf{I} can be shown algebraically; this is done in \cite{Thornblad2016c} in greater generality. A more intuitive argument might be the following. Since $(d_i)_{i=1}^n$ satisfies Moon's condition, there exists a generalised tournament with score sequence $(d_i)_{i=1}^n$. Now consider the blow--up of this tournament, formed by copying each of the $n$ vertices into $m$ identical vertices, letting each cluster of $m$ vertices form a regular sub--tournament (since $m$ is odd, the score for each vertex within each subtournament is $(m-1)/2$). This proves the existence of a generalised tournament with outdegrees $c_{i,\ell} = md_i+\frac{m-1}{2}$, implying that condition \textbf{I} must be satisfied.

Next we show that $c_{i,\ell}$ satisfies $c_{i,\ell} + c_{n+1-i,m+1-\ell} = mn-1$ for all $i=1,\dots, n$ and $\ell=1,\dots, m$. This corresponds precisely to condition \textbf{II}. Since $c_{i,\ell}$ is constant for $i$ fixed, we may take $\ell=1$. We have
\begin{align}
 c_{i,1} + c_{n+1-i,m} = m(d_i+d_{n+1-i})+m-1 = m(n-1)+m-1 = mn-1,
\end{align}
so condition \textbf{II} is satisfied.

By Theorem \ref{thm:Eplett}, there exists a (non--generalised) self--converse tournament $H$ on $mn$ vertices with outdegrees $c_{i,\ell}$. Denote by $v_{i,\ell}$ the vertex of $H$ with outdegree $c_{i,\ell}$. Let $\rho$ be an isomorphism $H\to H'$. By the proof of Theorem \ref{thm:Eplett} in \cite{Eplett1979}, we may assume that the cycle decomposition of $\rho$ consists of $\lfloor mn/2 \rfloor$ transpositions and a single fixed point (which must be a vertex with outdegree $c_{\lceil n/2 \rceil,\cdot} = (mn-1)/2$). In other words, we may assume that
\begin{align}
 \rho(v_{i,\ell})=v_{n+1-i,m+1-\ell}
\end{align}
for all $i=1,\dots, n$ and $\ell=1,\dots, m$.

We define now a generalised tournament $G=(V(G),\alpha)$ on $n$ vertices $w_1,\dots, w_n$ as follows. For $i,j=1,2,\dots, n$, let
\begin{align}
 \alpha(w_i,w_j) = \frac{1}{m^2}\sum_{\ell=1}^m \sum_{k=1}^m \alpha_H(v_{i,\ell},v_{j,k})
\end{align}
where $\alpha_H$ denotes the edge weight function of $H$ (which is an indicator function and can only take values in $\{0,1\}$). Note that
\begin{align}
 \alpha(w_i,w_j)+\alpha(w_j,w_i) = \frac{1}{m^2} \sum_{\ell=1}^m \sum_{k=1}^m ( \alpha_H(v_{i,\ell},v_{j,k})+\alpha_H(v_{j,k},v_{i,\ell}) = 1
\end{align}
so $\alpha$ is a valid weight function, i.e. $G$ is well--defined. We claim that $G$ has score sequence $(d_i)_{i=1}^n$. To see this,
\begin{align}
 \sum_{\substack{j=1 \\ j\neq i}}^n \alpha(w_i,w_j) 
& = \frac{1}{m^2}\sum_{\substack{j=1 \\ j\neq i}}^n \sum_{\ell=1}^m \sum_{k=1}^m \alpha_H(v_{i,\ell},v_{j,k}) \\
& = \frac{1}{m^2} \sum_{\ell=1}^m \sum_{\substack{i=1 \\ i\neq j}}^n \sum_{k=1}^m \alpha_H(v_{i,\ell},v_{j,k}) \\
& = \frac{1}{m^2} \sum_{\ell=1}^m md_i \\
&= d_i.
\end{align}

Finally we need to show that $G$ is self--converse. Let $\rho_G:V(G)\to V(G')$ be the bijection $\rho_G(w_i)=w_{n+1-i}$. It suffices to show that $\alpha(w_i,w_j)=1-\alpha(w_{n+1-i},w_{n+1-j})$, the latter being equal to $\alpha'(\rho_G(w_i),\rho_G(w_j))$. 

Using the fact that $H$ is self--converse, we have, for any $i\neq j$,
\begin{align}
 \alpha(w_i,w_j) 
& = \frac{1}{m^2}\sum_{\ell=1}^m \sum_{k=1}^m \alpha_H(v_{i,\ell},v_{j,k}) \\
& = \frac{1}{m^2}\sum_{\ell=1}^m \sum_{k=1}^m \left(1-\alpha_H(v_{n+1-i,m+1-\ell},v_{n+1-j,m+1-k})\right) \\
& = 1- \frac{1}{m^2} \sum_{\ell=1}^m \sum_{k=1}^m \alpha_H(v_{n+1-i,m+1-\ell},v_{n+1-j,m+1-k})\\
& = 1 - \alpha(w_{n+1-i},w_{n+1-j}).
\end{align}
Hence $G$ and $G'$ are isomorphic, so $G$ is self--converse. This completes the proof.
\end{proof}

\begin{proof}
Let
\begin{align}
 n' = \argmax\{ k\in\{1,2,\dots, \lfloor n/2 \rfloor\} \ : \ d_{n'}<(n-1)/2\}.
\end{align}
(We may assume this exists; if not, then all scores are equal to $(n-1)/2$ and hence rational, so no approximation is necessary.)

Define first $d_{k}^{(m)}=d_k=(n-1)/2$ for all $k=n'+1,\dots, \lceil n/2 \rceil$. Pick some rational $d_{n'}^{(m)}$ in the interval $(d_{n'},\min\{ (n-1)/2 ,d_{n'}+1/m\})$. Proceed inductively; having picked $d_{k}^{(m)}$ for some $1<k\leq n'$ we pick a rational $d_{k-1}^{(m)}$ in the interval $(d_{k-1},\min\{d_k^{(m)},d_{k-1}+1/m\})$. Proceed until we have picked rationals $d_{n'}^{(m)},\dots, d_1^{(m)}$. For $i=\lceil n/2 \rceil+1,\dots, n$, we define $d_i^{(m)}=n-1-d_{n+1-i}^{(m)}$. 

By construction we have $d_1^{(m)}\leq d_2^{(m)} \leq \dots \leq d_n^{(m)}$ and that condition \textbf{II} is met. To see that condition \textbf{I} is met, note that
\begin{align}
 \sum_{i=1}^{k} d_i^{(m)} \geq \sum_{i=1}^k d_i \geq \binom{k}{2}
\end{align}
for any $k=1,2,\dots, \lfloor n/2 \rfloor$, which is enough by the observations after the statement of the Lemma in Section \ref{sec:intro} By construction we have $|d_i-d_i^{(m)}|<1/m$ for each $i=1,2,\dots, n$, so $d_i^{(m)}\to d_i$ as $m\to \infty$.
\end{proof}

\begin{proof}[Proof of Theorem \ref{thm:real}]
Let $(d_i)_{i=1}^n$ be a sequence of non--negative reals satisfying conditions \textbf{I} and \textbf{II}. By Lemma \ref{lem:tech}, for each $=1,2,\dots, n$, we can find rationals $d_i^{(m)}$ converging to $d_i$ as $m\to \infty$, such that $(d_i^{(m)})_{i=1}^n$ satisfies conditions \textbf{I} and \textbf{II} for each $m\geq 1$.

By Lemma \ref{lem:rational}, there exists tournaments $G_m=(V(G_m),\alpha_m)$ with score sequences $(d_i^{(m)})_{i=1}^n$, respectively. We may assume that $V(G_m)=\{1,2,\dots, n\}$, that vertex $i$ has outdegree $d_i^{(m)}$ for each $m\geq 1$ and that $\rho : \{1,2,\dots, n\} \to \{1,2,\dots n\}$ defined by $\rho(i)=n+1-i$ is an isomorphism between $G_m$ and its converse $G_m'$.

The edge weights of each tournament are defined by the numbers $\{\alpha_m(i,j) \ : \ i,j=1,\dots, n, \ i\neq j\}$ which may be seen as an element of the compact set $[0,1]^{\binom{n}{2}}$. By passing to a subsequence, we may assume that $\alpha_m(i,j)$ converges as $m\to \infty$, for all $i,j=1,2,\dots, n$.

Let $G=(V(G),\alpha)$ be the generalised tournament with $V(G)=\{1,2,\dots, n\}$ and $\alpha(i,j)=\lim_{m\to \infty}\alpha_m(i,j)$. We should verify that this is a well--defined generalised tournament, that it has the appropriate score sequence, and that it is self--converse.

To see that it is well--defined, note that
\begin{align}
 \alpha(i,j) + \alpha(j,i) = \lim_{m\to \infty}\alpha_m(i,j) + \lim_{m\to \infty}\alpha_m(j,i) = \lim_{m\to \infty}(\alpha_m(i,j)+\alpha_m(j,i)) = 1,
\end{align}
for any $i,j=1,2,\dots, n$ with $i\neq j$, so $G$ is well--defined. (Also $\alpha(i,j)\in [0,1]$ and $\alpha(i,i)=1$.) By construction it holds that the score sequence of $G$ is $(d_i)_{i=1}^n$. Finally we claim that $\rho:\{1,2, \dots, n\}\to \{1,2,\dots, n\}$ defined by $\rho(i)=n+1-i$ is an isomorphism between $G$ and $G'$. To see this, note that
\begin{align}
 \alpha(i,j) = \lim_{m\to \infty}\alpha_m(i,j) = \lim_{m\to \infty}\alpha_m'(\rho(i),\rho(j)) 
&= \lim_{m\to \infty}(1-\alpha_m(\rho(i),\rho(j))) \\
&= 1-\lim_{m\to \infty}\alpha_m(\rho(i),\rho(j)) \\
&= 1-\alpha(\rho(i),\rho(j)) \\
&= \alpha'(\rho(i),\rho(j)). 
\end{align}
This completes the proof.

\end{proof}

\bibliographystyle{abbrv}

\begin{thebibliography}{1}

\bibitem{Eplett1979}
W.~J.~R. Eplett.
\newblock Self--converse tournaments.
\newblock {\em Canadian Mathematical Bulletin}, 22:23--27, 1979.

\bibitem{Landau1953}
H.~G. Landau.
\newblock On dominance relations and the structure of animal societies: {III}
  {T}he condition for a score structure.
\newblock {\em Bulletin of Mathematical Biophysics}, 15(2):143--148, 1953.

\bibitem{Moon63}
J.~W. Moon.
\newblock An extension of {L}andau's theorem on tournaments.
\newblock {\em Pacific Journal of Mathematics}, 13(4):1343--1345, 1963.

\bibitem{Thornblad2016c}
E.~Th{\"o}rnblad.
\newblock Another proof of {M}oon's theorem on generalised tournament score
  sequences.
\newblock Preprint.

\end{thebibliography}

\end{document}